\documentclass[11pt]{amsart}
\usepackage{amsthm,amsfonts,amsmath,amssymb,enumerate}
\newtheorem{thm}{Theorem}[section]
\newtheorem{cor}[thm]{Corollary}

\theoremstyle{definition}
\newtheorem{defn}[thm]{Definition}
\newtheorem{rem}[thm]{Remark}
\newtheorem{exam}[thm]{Example}

\numberwithin{equation}{section}
\usepackage{graphicx}
\usepackage[
 a4paper,
 left=2.5cm,
 right=2.5cm,
 top=3cm,
 bottom=3 cm,
 ]{geometry}

\begin{document}
\title[A generalization of \'Ciri\'c fixed point theorems]{A generalization of \'Ciri\'c fixed point theorems}

\author{Nguyen Van Dung}
\address[]{Faculty of Mathematics and Information Technology Teacher Education, Dong Thap University, Cao Lanh, Dong Thap 871200, Viet Nam}
\email{nvdung@dthu.edu.vn, nguyendungtc@yahoo.com}

\author{Poom Kumam}
\address[]{Department of Mathematics, Faculty of Science, King Mongkut's University of Technology Thonburi (KMUTT),
Bang Mod, Thrung Khru, Bangkok 10140, Thailand}
\email{poom.kum@kmutt.ac.th}

\author{Kanokwan Sitthithakerngkiet}
\address[]{Department of Mathematics, Faculty of Applied Science, King Mongkut's University of North Bangkok (KMUTNB), Wongsawang, Bangsue, Bangkok 10800, Thailand}
\email{kanokwans@kmutnb.ac.th}


\begin{abstract}

In this paper, we state and prove  a generalization of \'Ciri\'c fixed point theorems in metric space by using a new generalized quasi-contractive map. These theorems extend other well known fundamental metrical fixed point theorems in the literature (Banach \cite{SB1922}, Kannan \cite{RK1968}, Nadler \cite{SBN1969}, Reich \cite{SR1972}, etc.) Moreover, a multi-valued version for generalized quasi-contraction is also established.

\end{abstract}

\maketitle

\section{Introduction}

The  Banach's contraction principle \cite{SB1922} which was first appeared in 1922 is one of the most useful and important theorems in classical functional analysis. Its utility is not only to prove that,  in a complete metric space $X$, the contraction map $T$ (i.e., $d(Tx,Ty) \leq \alpha d(x,y)$ for some $0 \leq \alpha <1$ and for all $x,y \in X$)
has a unique fixed point  but also to show that the Picard iteration  converges to the fixed point.  For the reason that the contraction must be continuous, there are many researchers establish the fixed point theorems on various classes of  operators that are weaker than contractive conditions but are not continuous, see for example \cite{RK1968, SR1972}.

One of the most well-known results in generalizations of Banach's contraction principle which the Picard iteration still converges to the fixed point of map  is the \'Ciri\'c fixed point theorem \cite{LBC1974}. Before providing the \'Ciri\'c fixed point theorem, we recall that a  self-map  $T$ on a metric space $(X,d)$,    is said to be a \emph{quasi-contraction}  iff there exists a nonnegative number $q <1$ such that for all $x,y \in X$,
\begin{eqnarray} \label{139-64}
d(Tx,Ty) &\le& q   \max \big\{d(x,y), d(x,Tx), d(y,Ty), d(x,Ty), d(y,Tx) \big\}.
\end{eqnarray}
The \'Ciri\'c fixed point theorem is given by the following theorem.
\begin{thm} [\cite{LBC1974}, Theorem 1] \label{139-65}
Let the metric space $X$ be $T$-orbitally complete and let $T$ be a quasi-contraction.
Then we have
\begin{enumerate}
\item $T$ has a unique fixed point $ x ^*$ in $X$.

\item $\lim\limits_{n \rightarrow \infty} T ^nx = x ^*$ for all $x \in X$.

\item $ d( T ^ nx,x ^* ) \le \cfrac{ q ^ n}{1 - q} d(x,Tx) $ for all $x \in X$.
\end{enumerate}
\end{thm}

This result was generalized to many results, such as a common fixed point theorem of nonlinear contraction~\cite[Theorem~4]{SS1982}, a generalized $\varphi$-contraction \cite[Section~2.6]{VB2006}, a \'Ciri\'c almost contraction \cite[Theorem~3.2]{VB2008} and see also \cite{LBC2009, JS2009, LC2009}. But from the well-known result of Rhoades \cite{BER1977} in 1977 to recent surveys, in Berinde \cite{VB2006} and Collaco and Silva \cite{CS1997} for instance, there were no any other value added to quasi-contraction condition. On the other hand, the  Banach's contraction principle has been extended to multi-valued contractions by Nadler \cite{SBN1969} and see also \cite{DK1995, EAV2005, AAH2011, PVS2002}.

In this paper, we define a new generalized quasi-contraction by adding four new values $d(T ^2 x,x)$,  $d(T ^2 x,Tx)$, $d(T ^2 x,y)$, $d(T ^2 x,Ty)$ to a quasi-contraction condition. Also, an example is presented. After that we state and prove unique fixed point theorems which are the generalization of \'Ciri\'c fixed point theorem in~\cite{LBC1974}. Moreover, we also establish fixed point theorems for multi-valued  generalized quasi-contraction.

\section{Preliminaries}
First, we recall some notions which will be used in what follows. Let $(X,d)$ be a metric space and $A,B$ be any  two subsets of $X$. We denote
\begin{eqnarray*}  D(A,B) &=& \inf \big\{ d(a,b): a \in A, b\in B \big\}\\
\rho(A,B) & =& \sup \big\{d(a,b): a \in A, b\in B \big\}\\
BN(X) &=& \big\{A: \emptyset \neq A \subset X \text{ and } \delta (A) < +\infty  \big\},
\end{eqnarray*}
where $\delta(A) := \sup \big\{d(a,b): a,b \in A\big\}$.

\begin{defn}[\cite{LBC1974}] Let $T: X \longrightarrow X$ be a map on metric space.
For each $x \in X$ and for any positive integer $n$, put  $$ O_T(x,n) = \{x, Tx, \ldots, T ^ nx \}~and~  O_T(x,+\infty) = \{x, Tx, \ldots, T ^ nx, \ldots \}.$$
  The set $O_T(x,+\infty)$ is called  the \emph{orbit} of $T$ at $x$ and the metric space
$X$ is called \emph{$T$-orbitally complete} if every Cauchy sequence in $O_T(x,+\infty)$ is convergent in $X$.
\end{defn}

Note that every complete metric space is $T$-orbitally complete for all maps $T: X \longrightarrow X$.
The following example shows that there exists a $T$-orbitally complete metric space but it is not complete.

\begin{exam} Let $(X,d)$ be a metric space which is not complete and $T: X \longrightarrow X$ be the map defined by $Tx = x_0$ for all $x \in X$ and some $x_0 \in X$. Then $(X,d)$ is a $T$-orbitally complete metric space which is not complete.
\end{exam}

\begin{defn}[\cite{LBC1974}] \label{19-94}
Let $F: X \longrightarrow BN(X)$ be a multi-valued mapping. Let $x_0 \in X$, an \emph{orbit} of $F$ at $x_0$ is a sequence $$\big\{x_n: x_n \in F x_{n-1}, n \in \mathbb{N}\big\} .$$ A space $X$ is called to be \emph{$F$-orbitally complete} if every Cauchy sequence which is a subsequence of an orbit of $F$ at $x$ for some $x \in X$, converges in $X$.
\end{defn}

Next, the definitions of generalized quasi-contraction for single-valued and multi-valued are given as follows;

\begin{defn} \label{sgqc}
Let $T: X \longrightarrow X$ be a mapping on metric space $X$. The mapping $T$ is said to be a \emph{generalized quasi-contraction}  iff   there exists $q \in [0,1)$ such that for all $x, y \in X$,
\begin{eqnarray} \label{139-64}
d(Tx,Ty) &\le& q   \max \big\{d(x,y), d(x,Tx), d(y,Ty), d(x,Ty), d(y,Tx),\\ \nonumber
&&d(T ^2 x,x), d(T ^2 x,Tx) ,d(T ^2 x,y), d(T ^2x, Ty) \big\}.
\end{eqnarray}
\end{defn}

\begin{exam} \label{ex} Let $X = \{1,2,3,4,5 \} $ with $d$ defined as
$$ d(x,y) = \left\{\begin{array}{ll}0 &\hbox{ if } x = y\\ 2 &\hbox{ if } (x,y) \in \big\{(1,4), (1,5), (4,1), (5,1) \big \}\\
1 & \hbox{ otherwise.} \end{array}\right. $$
Let $T: X \longrightarrow X$ be defined by $$ T1 = T2 =T3 =1, T4 = 2, T5 =3.$$
Then, we have
$$d(Tx,Ty) = d(1,1) = 0 ~\text{ if } ~ x,y \in \{ 1,2,3\}; $$
$$d(T1,T4) =d(T2,T4) = d(T3,T4) = d(1,2) =1;$$
$$ d(T1,4) = d(T2,4) = d(T3,4) = d(1,4) = 2;$$
$$d(T1,T5) = d(T2,T5) = d(T3,T5) = d(1,3) = 1;$$
$$ d(T1,5) = d(T2,5) = d(T3,5) = d(1,5) =2;$$
$$d(T4,T5) = d(2,3) =1;$$
$$d(4,5) = d(4,T4) =d(5,T5) = d(4,T5) =d(5,T4) =1;$$
$$d(T ^2 4,4) = d(T2,4) = d(1,4) =2;$$
$$d(T ^2 5,5) = d(T3,5) = d(1,5) =2.$$

The above calculations show that $T$ is not quasi-contraction for $x =4$ and $y=5$ because there is no a nonnegative number $q<1$ satisfying the equation \eqref{139-64}.  However, $T$ is generalized quasi-contraction since the \eqref{139-64} holds for some $q \in \big[0.5,1 \big)$ and  for all $x,y \in X$.
\end{exam}

\section{The main results}
On the following results, we state and prove the new fixed point theorems which are general cases of the \'Ciri\'c fixed point theorem.

\begin{thm} \label{139-99} Let $(X,d)$ be a metric space. Suppose that $T: X \longrightarrow X$ is a generalized quasi-contraction and $X$ is $T$-orbitally complete.
Then we have
\begin{enumerate} \item \label{139-99-3} $T$ has a unique fixed point $ x ^*$ in $X$.

\item \label{139-99-4} $\lim\limits_{n \rightarrow \infty} T ^nx = x ^*$ for all $x \in X$.

\item \label{139-99-5} $ d( T ^ nx,x ^* ) \le \cfrac{ q ^ n}{1 - q} d(x,Tx)$ for all $x \in X$ and $n \in \mathbb{N}$.
\end{enumerate}
\end{thm}

\begin{proof}
\eqref{139-99-3}. \textbf{Step 1.} \emph{$T$ has a fixed point.} For each $x \in X$ and $1 \le i \le n-1$ and $ 1\le j \le n$, we have
\begin{eqnarray} \label{139-97} &&d(T ^ix,T ^ jx) \\ \nonumber
&=& d(T T ^{ i-1}x,T T ^{ j-1}x)\\ \nonumber
&\le &
q  \max \big\{
d(T ^{i-1}x,T ^{j-1}x), d(T ^{i-1}x,TT ^{i-1}x), d(T ^{j-1}x,TT ^{j-1}x), d(T ^{i-1}x,TT ^{j-1}x),\\ \nonumber
&&
d(T ^{j-1}x,TT ^{i-1}x),
d(T ^2 T ^{i-1}x,T ^{i-1}x), d(T ^2 T ^{i-1}x,TT ^{i-1}x), d(T ^2 T ^{i-1}x,T ^{j-1}x)\\ \nonumber
&&
 d(T ^2 T ^{i-1}x,TT ^{j-1}x) \big\}\\ \nonumber
&= &
q  \max \big\{
d(T ^{i-1}x,T ^{j-1}x), d(T ^{i-1}x,T ^{i}x), d(T ^{j-1}x,T ^{j}x), d(T ^{i-1}x,T ^{j}x),\\ \nonumber
&&
d(T ^{j-1}x,T ^{i}x),
d(T ^{i+1}x,T ^{i-1}x), d(T ^{i+1}x,T ^{i}x), d(T ^{i+1}x,T ^{j-1}x)\\ \nonumber
&&
 d(T ^{i+1}x,T ^{j}x) \big\}\\ \nonumber
&\le&
q  \delta \big[O_T(x,n)\big]
\end{eqnarray} where $\delta \big[O_T(x,n)\big] = \max \big\{d(T ^ ix,T ^ jx): 0 \le i,j \le n \big\} $.

From~\eqref{139-97}, since $0 \le q <1$, there exists $k_n(x) \le n$ such that
\begin{equation} \label{139-96} d(x,T ^{k_n(x)}x) = \delta \big[O_T(x,n) \big].
\end{equation}

Then we have
\begin{eqnarray*} d(x, T ^{k_n(x)}x) & \le & d(x, Tx) + d(Tx, T ^{k_n(x)}x)\\ \nonumber
&\le &
d(x,Tx) + q  \delta \big[O_T(x,n) \big] \\ \nonumber
&=&
d(x,Tx) + q  d(x,T ^{ k_n(x)}x).
\end{eqnarray*}

It implies that
\begin{equation} \label{139-95} \delta \big[O_T(x,n) \big] = d(x,T ^{ k_n(x)}x) \le \cfrac{ 1}{1 -q}d(x,Tx).
\end{equation}

For all $n,m \le 1$ and $ n <m$, it follows from the generalized quasi-contractive condition of $T$  and~\eqref{139-95} that

\begin{eqnarray} \label{139-92} d(T ^ nx,T ^ mx) &=& d (T T ^{n-1}x, T ^{m- n+1}T ^{ n-1}x)\\ \nonumber
&\le &
q  \delta \big[O_T(T ^{ n-1}x,m-n+1) \big] \\ \nonumber
&=&
q  d (T ^{n-1}x, T ^{k_{m-n+1}(T ^{ n-1}x)}T ^{ n-1}x)\\ \nonumber
&=&
q  d (TT ^{n-2}x, T ^{k_{m-n+1}(T ^{ n-1}x)+1}T ^{ n-2}x)\\ \nonumber
&\le &
q ^2   \delta \big[O_T(T ^{n-2}x,k_{m-n+1}(T ^{ n-1}x)+1) \big] \\ \nonumber
&\le&
q ^2  \delta \big[O_T(T ^{ n-2}x,m-n+2) \big] \\ \nonumber
&\le &
\ldots \\ \nonumber
&\le & q ^n \delta \big[O_T(x,m) \big] \\ \nonumber
&\le &
\cfrac{ q ^ n}{1 -q} d(x,Tx).
\end{eqnarray}

Since $\lim\limits_{n \rightarrow \infty} q ^ n = 0$, $\{T ^ nx \}$ is a Cauchy sequence in $X$. Since $X$ is $T$-orbitally complete, there exists $x^* \in X$ such that
\begin{equation} \label{139-91} \lim\limits_{n \rightarrow \infty} T ^ nx = x ^*.
\end{equation}

By using the generalized quasi-contractive condition of $T$  again, we have
\begin{eqnarray} \label{139-87} &&d(x ^* ,T x ^* )\\ \nonumber
& \le & d(x ^*, T ^{ n+1}x) + d(T ^{ n+1}x, Tx ^* )\\ \nonumber
&= & d(x ^*, T ^{ n+1}x ) + d(TT ^ nx,Tx ^*)\\ \nonumber
&\le &
d(x ^*, T ^{ n+1}x ) + q \max \big\{d(T ^ nx,x ^*), d(T ^ nx,TT ^ nx), d(x ^*,Tx ^*),
d(T ^ nx, Tx ^*), \\ \nonumber
&&
d(x ^*,TT ^ nx),  d(T ^2 T ^nx,T ^ nx), d(T ^2 T ^ nx,TT ^ nx), d(T ^2 T ^ nx,x ^*), d(T ^2 T ^ nx,Tx ^*) \big\}\\ \nonumber
&= &
d(x ^*, T ^{ n+1}x ) + q \max \big\{d(T ^ nx,x ^*), d(T ^ nx,T ^{ n+1}x), d(x ^*,Tx ^*),
d(T ^ nx, Tx ^*), \\ \nonumber
&&
d(x ^*,T ^{n+1}x),  d(T ^{n+2}x,T ^ nx), d(T ^{ n+2}x,T ^{ n+1}x), d(T ^{ n+2}x,x ^*), d(T ^{n+2}x,Tx ^*) \big\}.
\end{eqnarray}
Taking the limit as $n \rightarrow \infty $ in~\eqref{139-87}, and using~\eqref{139-91}, we get $ d(x ^*,Tx ^*) \le q d(x ^*,Tx ^*) .
$
Since $q \in [0,1)$, we obtain $d(x ^*,Tx ^*) = 0$, that is, $x ^* =Tx ^*$. Then $T$ has a fixed~point.

\textbf{Step 2.} \emph{The fixed point of
$T$ is unique.}
Let $x ^*, y ^*$ be two fixed points of $T$. Since  $T$ is generalized quasi-contraction, we have
\begin{eqnarray*} \label{139-93} d(x ^*,y ^*)
&=&
d(Tx ^*,Ty^*) \\
&\le& q  \max \big\{d(x^*,y^*), d(x^*,Tx^*), d(y^*,Ty^*), d(x^*,Ty^*), d(y^*,Tx^*),\\
&&d(T ^2 x ^*, x ^*), d(T ^2 x^*,Tx^*),d(T ^2 x^*,y^*), d(T ^2 x^*,Ty^*) \big\}\\
&=&
qd(x ^*,y ^*).
\end{eqnarray*}
Since $q \in [0,1)$, we obtain $d(x ^*,y ^*) =0$.
That is, $x ^* = y ^*$. Then the fixed point of $T$ is unique.

\eqref{139-99-4}. It is proved by~\eqref{139-91}.

\eqref{139-99-5}. Taking the limit as $m \rightarrow \infty$ in~\eqref{139-92}, we get $d( T ^ nx,x ^* ) \le \cfrac{ q ^ n}{1 - q}d(x,Tx).$
\end{proof}

\begin{cor} \label{19-99} Let $(X,d)$ be a metric space and $T: X \longrightarrow X$ be a map satisfying the following:
\begin{enumerate} \item \label{19-99-1} $X$ is $T$-orbitally complete.

\item \label{19-99-2} There exists $k \in \mathbb{N}$ and $q \in [0,1)$ such that for all $x, y \in X$,
\begin{eqnarray} \label{19-98} d(T^kx,T^ky)
&\le& q  \max \big\{d(x,y), d(x,T^kx), d(y,T^ky), d(x,T^ky), d(y,T^kx),\\ \nonumber
&&d(T ^{2k} x,x), d(T ^{2k} x,T^kx) ,d(T ^{2k} x,y), d(T ^{2k}x, T^ky) \big\}.
\end{eqnarray}
\end{enumerate}
Then we have
\begin{enumerate} \item \label{19-99-3} $T$ has a unique fixed point $ x ^*$ in $X$.

\item \label{19-99-5} $ d( T ^ nx,x ^* ) \le \cfrac{ q ^ m}{1 - q} \max \big\{d(T ^ ix,T ^{ i+k}x): i =0,1, \ldots ,k-1 \big\} $ for all $x \in X$ and $n \in \mathbb{N}$ where $m$ is the greatest integer not exceeding $\cfrac{ n}{k}$.

\item \label{19-99-4} $\lim\limits_{n \rightarrow \infty} T ^nx = x ^*$ for all $x \in X$.
\end{enumerate}
\end{cor}

\begin{proof} \eqref{19-99-3}. By the conclusion of Theorem~\ref{139-99}, $ T ^ k$ has a unique fixed point $x ^*$ and $T ^ k(Tx ^*) = T (T ^ kx ^*) = Tx ^*$. It implies that $Tx ^* =x ^*$, that is, $T$ has a fixed point $x ^*$.
The uniqueness of the fixed point of $T$ is easy to see.

\eqref{19-99-5}. Let $n \in \mathbb{N}$. Then $n = mk +j$, $0 \le j <k$ and for each $x \in X$, $T ^ nx = (T ^ k)^mT ^ jx$. It follows from Theorem~\ref{139-99}.\eqref{139-99-5} that
\begin{eqnarray*} d(T ^ nx,x^*) & \le & \cfrac{ q ^ m}{1 -q}d(T ^ jx,T ^ kT ^ jx)\\
& \le &
\cfrac{ q ^ m}{1 -q} \max \big\{d(T ^ ix,T ^{ i+k}x): i =0,1, \ldots ,k-1 \big\} .
\end{eqnarray*}

\eqref{19-99-4}. It is a direct consequence of  \eqref{19-99-5}.
\end{proof}


\begin{cor}[\cite{LBC1974}, Theorem~2] \label{19-93} Let $(X,d)$ be a metric space and $T: X \longrightarrow X$ be a map satisfying the following:
\begin{enumerate} \item \label{19-93-1} $X$ is $T$-orbitally complete.

\item \label{19-93-2} There exists $k \in \mathbb{N}$ and $q \in [0,1)$ such that for all $x, y \in X$,
\begin{eqnarray} \label{19-92} d(T^kx,T^ky)
&\le& q  \max \big\{d(x,y), d(x,T^kx), d(y,T^ky), d(x,T^ky), d(y,T^kx) \big\}.
\end{eqnarray}
\end{enumerate}
Then we have
\begin{enumerate} \item \label{19-93-3} $T$ has a unique fixed point $ x ^*$ in $X$;

\item \label{19-93-5} $ d( T ^ nx,x ^* ) \le \cfrac{ q ^ m}{1 - q} \max \big\{d(T ^ ix,T ^{ i+k}x): i =0,1, \ldots ,k-1 \big\} $ for all $x \in X$ and $n \in \mathbb{N}$ where $m$ is the greatest integer not exceeding $\cfrac{ n}{k}$;

\item \label{19-93-4} $\lim\limits_{n \rightarrow \infty} T ^nx = x ^*$ for all $x \in X$.
\end{enumerate}
\end{cor}

Now, we denote the multi-valued mapping $F: X \longrightarrow BN(X)$ of generalized quasi-contraction by
\begin{eqnarray} \label{19-96} \rho(Fx,Fy)
&\le& q  \max \big\{d(x,y), \rho(x,Fx), \rho(y,Fy), D(x,Fy), D(y,Fx),\\ \nonumber
&&D(F^2 x,x), D(F ^2 x,Fx) ,D(F ^2 x,y), D(F ^2x, Fy) \big\},
\end{eqnarray}
for some $q \in [0,1)$ and for all $x, y \in X$. The following theorem presents the fixed point theorem for multi-valued version of generalized quasi-contractive mapping.

\begin{thm} \label{19-97} Let $(X,d)$ be a metric space and $F: X \longrightarrow BN(X)$ be a multi-valued map. Suppose that $F$ is a generalized quasi-contraction
and $X$ is $F$-orbitally complete.
Then we have
\begin{enumerate} \item \label{19-97-3} $F$ has a unique fixed point $ x ^*$ in $X$ and $F x^* = \{ x^*\}$.
\item \label{19-97-4} For each $x_0 \in X$, there exists an orbit $\{ x_n\}_n$ of $F$ at $x_0$ such that  $\lim\limits_{n \rightarrow \infty} x_n = x ^*$ for all $x \in X$, and
\item \label{19-97-5} $ d( x_n,x ^* ) \le \cfrac{ (q ^{ 1-a}) ^ n}{1 - q ^{ 1-a}} d(x_0,x_1)$ for all $n \in \mathbb{N}$, where $a<1$ is any fixed positive number.
\end{enumerate}
\end{thm}

\begin{proof} \eqref{19-97-3}. Given $a \in (0,1)$ and defined a single-valued mapping $T: X \longrightarrow X$ by the following statement:
$$ for~ each~ x \in X,  ~Tx \in  Fx ~ ~ satisfies ~~d(x,Tx) \ge q ^ a \rho (x,Fx).$$
 By the Definition \ref{19-94} and the condition of $F$, we have for every $x,y \in X$,
\begin{eqnarray*} d(Tx,Ty) & \le & \rho(Fx,Fy) \\
& \le &
q\max \big\{d(x,y), \rho(x,Fx), \rho(y,Fy), D(x,Fy), D(y,Fx),\\ \nonumber
&&D(F^2 x,x), D(F ^2 x,Fx) ,D(F ^2 x,y), D(F ^2x, Fy) \big\}\\
&= &
qq ^{ -a}\max \big\{q ^ad(x,y), q ^a \rho(x,Fx), q ^a\rho(y,Fy), q ^aD(x,Fy), q ^aD(y,Fx),\\ \nonumber
&&q ^aD(F^2 x,x), q ^aD(F ^2 x,Fx) ,q ^aD(F ^2 x,y), q ^aD(F ^2x, Fy) \big\}\\
&\le &
q^{1 -a}\max \big\{d(x,y), d(x,Tx), d(y,Ty), d(x,Ty), d(y,Tx),\\ \nonumber
&&d(T^2 x,x), d(T^2 x,Tx) ,d(T^2 x,y), d(T^2x, Ty) \big\}.
\end{eqnarray*}
By Theorem~\ref{139-99}, we conclude that $T$ has a unique fixed point $x ^*$. Then $\rho(x ^*,F x^*) \le q ^ a d(x^*,T x^*) = 0$ implies that $\rho(x^*, F x^*) = 0$. Then $x^*$ is a fixed point of $F$ and $F x^* = \{x^* \}$.
From the direct consequences of Theorem~\ref{139-99} where $x_n = T ^ nx$ for all $n \in \mathbb{N}$, we obtain that
\eqref{19-97-4} and \eqref{19-97-5} hold.
\end{proof}

\begin{cor}[\cite{LBC1974}, Theorem~3] \label{19-91} Let $(X,d)$ be a metric space and $F: X \longrightarrow BN(X)$ be a multi-valued map satisfying the following:
\begin{enumerate} \item \label{19-91-1} $X$ is $F$-orbitally complete.

\item \label{19-91-2} There exists $q \in [0,1)$ such that for all $x, y \in X$,
\begin{eqnarray} \label{19-90} \rho(Fx,Fy)
&\le& q  \max \big\{d(x,y), \rho(x,Fx), \rho(y,Fy), D(x,Fy), D(y,Fx) \big\}.
\end{eqnarray}
\end{enumerate}
Then we have
\begin{enumerate} \item \label{19-91-3} $F$ has a unique fixed point $ x ^*$ in $X$ and $F x^* = \{ x^*\}$.

\item \label{19-91-4} For each $x_0 \in X$, there exists an orbit $\{ x_n\}_n$ of $F$ at $x_0$ such that  $\lim\limits_{n \rightarrow \infty} x_n = x ^*$ for all $x \in X$, and

\item \label{19-91-5} $ d( x_n,x ^* ) \le \cfrac{ (q ^{ 1-a}) ^ n}{1 - q ^{ 1-a}} d(x_0,x_1)$ for all $n \in \mathbb{N}$, where $a<1$ is any fixed positive number.
\end{enumerate}
\end{cor}

\begin{exam} \label{19-95} Let $(X,d)$ and $T: X \longrightarrow X$ be defined by Example \ref{ex}.

It is easy to see that  $X$ is $T$-orbitally complete metric space. By the definition of the distance  $d$ and mapping $T$ , we conclude that $X$ and $T$ satisfy all of the conditions in Theorem \ref{139-99}.  Clearly, $x^{*} = 1$
is a unique fixed point of $T$.

Note that, if $x \in \{1,2,3 \}$ then $T^{n}x = 1 $ for $ n= 1,2,3,...$ and
if $x \in \{4,5 \}$ then $T^{n}x = 1 $ for $ n= 2,3,4,...$. That is $\lim\limits_{n \rightarrow \infty} T^{n}x  = x ^*$ for all $x \in X$.

Let $q \in \big[ 0.5,1 \big)$ be fixed by the generalized quasi-contraction of $T$ which arises from Example \ref{ex}. We see that the inequality $ d( T ^ nx,x ^* ) \le \cfrac{ q ^ n}{1 - q} d(x,Tx) $ holds for all $x \in X$ and $n \in \mathbb{N}$.

Therefore, this example is presented to certify the results of  Theorem~\ref{139-99}. However, it is not  applicable to Theorem \ref{139-65}.
\end{exam}

\begin{rem} Example~\ref{19-95} shows that our results are proper generalizations of \'Ciri\'c fixed point theorems in~\cite{LBC1974}. Then our results are exactly a new form of fixed point theorems in metric spaces. Moreover, we may generalize other fixed point theorems contained at most five mentioned values in the literature to that contain $d(T ^2 x,x)$, $d(T ^2 ,Tx)$, $d(T ^2 x,y)$, $d(T ^2 x,Ty)$ in addition.
\end{rem}

\noindent{\underline{\bf Acknowledgement.}}
This research of is supported by King Mongkut's University
of Technology North Bangkok, Thailand


\end{document}